\documentclass[12pt]{amsart}

\usepackage{graphicx}
\usepackage[fontsize=14pt]{scrextend}
\usepackage{mathrsfs}
\usepackage[]{url}

\makeatletter
\renewcommand\@biblabel[1]{}
\makeatother

\newtheorem{theorem}{Theorem}
\newtheorem{lemma}[theorem]{Lemma}
\newtheorem{prop}[theorem]{Proposition}
\newtheorem{conj}[theorem]{Conjecture}

\theoremstyle{definition}

\theoremstyle{remark}
\newtheorem{remark}[theorem]{Remark}
\theoremstyle{remark}

\numberwithin{equation}{theorem}
\numberwithin{equation}{section}
\numberwithin{theorem}{section}

\numberwithin{equation}{section}

\newcommand{\R}{\mathbb{R}}

\title[The Log Convex Density Conjecture in Hyperbolic Space]{The Log Convex Density Conjecture in Hyperbolic Space}
\author[\tiny{L. Di Giosia,  J. Habib,  L. Kenigsberg,  D. Pittman, W. Zhu}]{Leonardo Di Giosia,  Jahangir Habib,  Lea Kenigsberg,  Dylanger Pittman,  Weitao Zhu\vspace{-5ex}}
\begin{document}
\begingroup
\def\uppercasenonmath#1{}
\let\MakeUppercase\relax 
\maketitle
\endgroup








\begin{abstract}

The Euclidean log convex density theorem, proved by Gregory Chambers in 2015, asserts that in Euclidean space with a log-convex density spheres about the origin are isoperimetric. We provide a partial extension to hyperbolic space in which volume and perimeter densities are related but different.

\end{abstract}


\section{Introduction}

The Euclidean log convex density theorem, proved by Gregory Chambers [Ch] in 2015, asserts that in Euclidean space with a log convex density spheres about the origin are isoperimetric (minimize perimeter for given volume). The analogous conjecture in hyperbolic space, $\mathbb{H}^n$, says that for a log convex density, spheres about the origin are isoperimetric. We prove an easier version in which the perimeter is weighted by a larger density than the volume:

\vspace{.5cm}
\textbf{Theorem \ref{weakHyperbolic}}. \textit{Consider $\mathbb{H}^n$ with smooth, radial, log convex volume density $\phi(R)$ and perimeter density} $$\phi(R) \cdot2\cosh^2(R/2),$$

\textit{where $R$ is distance from the origin. Then spheres about the origin are uniquely isoperimetric.} 

\vspace{.5cm}
Note that a smooth radial density is log convex if and only if it is a log convex function of $\mathbb{R}$.

The problem of perimeter minimization is an ancient one. Two thousand years ago, Zenodorus showed that planar circles are perimeter minimizers for given enclosed area. Similar results were proven over the past 150 years for spheres in $\mathbb{R}^n$, $\mathbb{S}^n$, and $\mathbb{H}^n$ as well (see [Mo, \textsection 13.2]).

More recently, mathematicians have studied manifolds with density. A density $f$ is a positive function defined on a manifold. For a region $\Omega$, the weighted volume of $\Omega$ with respect to density $f$ is $$\int_{\Omega}f\,dV,$$ and the weighted area of its boundary is $$\int_{\partial \Omega} f\,dA.$$ The interest in manifolds with density is partially due to its role in Perelman's 2002 proof of the Poincar\'{e} conjecture [Mo, Chapt. 18].

\subsection{Proof of Theorem \ref{weakHyperbolic}}
Two important steps are needed in establishing this result. 

Firstly, we use the Poincar\'{e} ball model of hypebolic space to reduce the problem to the open ball in $\mathbb{R}^n$ with equal volume and perimeter densities which diverge at the boundary.

Secondly, Proposition \ref{noBoundary} shows that a component of an isoperimetric region must be bounded. Otherwise we carefully truncate the region and restore the volume elsewhere to get a good candidate which nonetheless must have no less perimeter. This inequality, paired with other differential inequalities relating the boundary of the truncation to the rate of volume growth, forces the volume outside the ball of radius $r$ to reach 0 in finite time.



\subsection{Paper outline}
Section 2 provides notation. Section 3 begins with an explicit formula (Lemma \ref{tanh}) relating Poincar\'{e} ball and exponential coordinates on $\mathbb{H}^n$. Proposition \ref{equiv} shows that our unequal densities on $\mathbb{H}^n$ are equivalent to equal densities on the Euclidean ball. Propositions \ref{ExistNBall} and \ref{noBoundary} after Morgan-Pratelli [MP, Thms. 3.3, 5.9] establish the existence and boundedness of isoperimetric regions. Now the local analysis of Chambers [Ch] for equal densities on Euclidean space completes the proof of our main Theorem \ref{weakHyperbolic}.

\subsection{Acknowledgements}
We would like to thank the NSF REU grant program, Williams College (including the Finnerty Fund), Stony Brook University and the Mathematical Association of America for funding. Additionally,  we would like to thank our advisor Frank Morgan for valuable input and discussions throughout our research.

\section{Notation}
Here we provide a reference for  the notation used in many of our lemmas and propositions.

     \begin{figure}[h!]\label {figurr}

  \includegraphics[width=.9
\textwidth]{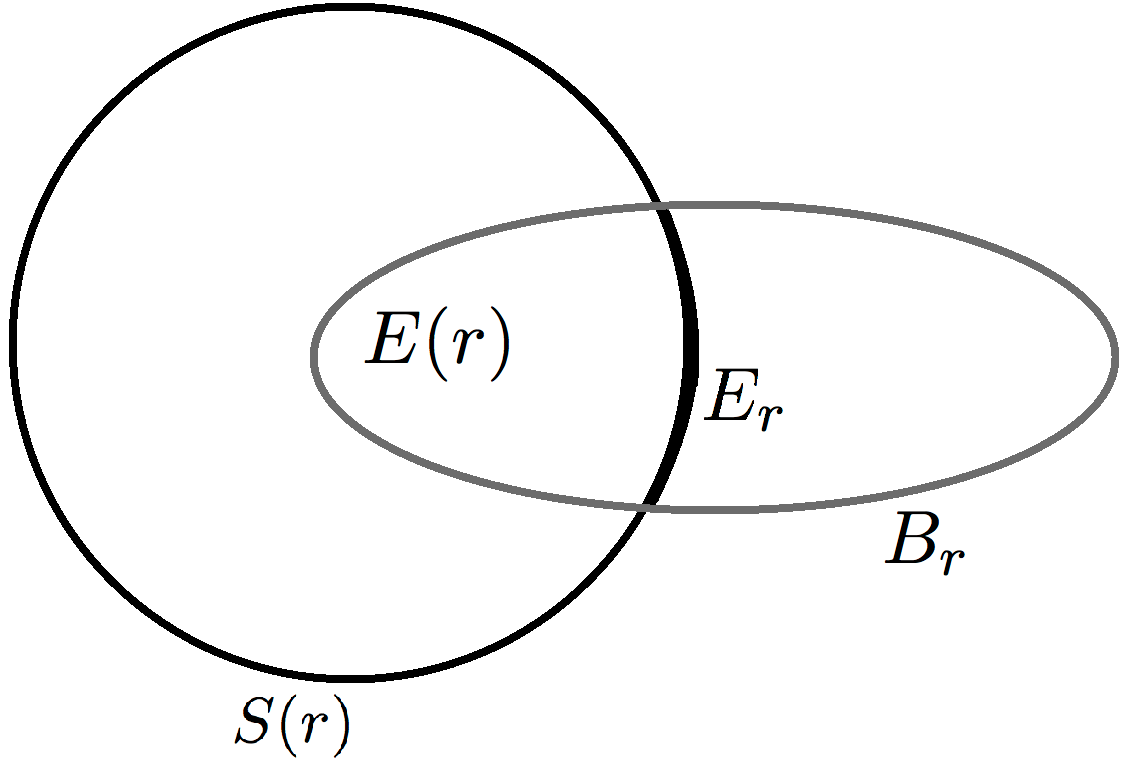}

\caption{By relating the slice $E_r$ of $E$ to the rate of growth of the volume $|E(r)|$ inside the unit sphere, one can obtain inequalities useful in proving existence and boundedness.}
\end{figure}

As in Figure 1, in both hyperbolic and Euclidean space, denote the sphere and ball of radius $r$ by $S(r)$ and $B(r)$. Let $|S|$ denote the unweighted measure of a surface or region $S$. For a region $E$, let $E_r$ be the slice of $E$ by the sphere $S(r)$ (see [Mo, 4.11]). Let $B_r$ be the restriction of the boundary of $E$ to the exterior of the ball $B(r)$ and define $P(r)$ as $|B_r|$. Let $E(r)$ denote the restriction of $E$ to the ball $B(r)$.
 Furthermore, let $V(r)$ be the volume of the restriction of $E$ to the exterior of the ball $B(r)$. Let  $p(r)$ denote the perimeter of $E_r$. In the presence of any density $f$, apply $f$ as a subscript to indicate weighted volume or perimeter.






\section{Proof of Theorem}
After two lemmas, Proposition \ref{equiv} equates hyperbolic space with unequal densities to the Euclidean ball with equal densities. Propositions \ref{ExistNBall} and \ref{noBoundary} prove existence and boundedness. Our main Theorem \ref{weakHyperbolic} now can apply Chambers [Ch] to prove spheres uniquely isoperimetric in $\mathbb{H}^n$ with our unequal densities.

\begin{lemma}\label{tanh}
On the Poincar\'{e} ball model of $\mathbb{H}^n$, distance from the origin is $2\tanh^{-1}{r}$.
\end{lemma}
\begin{proof}
This distance may be written as

\begin{align*}
\int_0^{r}\frac{2}{1-x^2}\,dx &= \int_0^{r}\frac{1}{1-x}+\frac{1}{1+x}\,dx \\
&= \ln{\bigg(\frac{1+r}{1-r}\bigg)} \\
&= 2\tanh^{-1}{r}.
\end{align*}
\end{proof}

Thus, a function $\phi:[0, 1)\to(0, \infty)$ can be considered a radial density on hyperbolic space as a function of distance from the origin by computing the composition $\phi\circ\tanh{(r/2)}$. This allows for comparison between densities on hyperbolic space and densities on the Euclidean ball. Lemma \ref{convex} establishes that log convexity of $\phi\circ \tanh{(r/2)}$ on $[0,\infty)$ implies log convexity of $\phi$ on $[0,1)$

 \begin{lemma}\label{convex}

Let $g$ be a smooth function with $g'$ positive and $g''$ negative. If the composition of a smooth positive nondecreasing function $f$ with $g$ is log convex, then $f$ is log convex.
 \end{lemma}
 \begin{proof}

  By assumption,

	$$0 \leq f^2\cdot(\log(f\circ g))''=(g')^2\cdot(ff''-(f')^2)+ff'g''$$
    
    $$\leq(g')^2\cdot(ff''-(f')^2) = (g'\cdot f)^2(\log f)''$$ Therefore $f$ is log convex.

\end{proof}
Take any smooth radial density on hyperbolic space. Let $\phi(r)$ be the density as a function of Poincar\'{e} radial coordinate $r$. Then if the density is log convex, $\phi$ must be nondecreasing and when composed with $\tanh{(r/2)}$ is log convex on $[0, \infty)$. Taking $g$ to be $\tanh{(r/2)}$ provides that $\phi$ is log convex on $[0, 1)$.  

As previously mentioned, the hyperbolic metric in Poincar\'{e} coordinates is a conformal rescaling of that on the Euclidean ball. By the effects of conformally rescaling a metric, we see that hyperbolic space has the same isoperimetric profile in Poincar\'{e} coordinates as the open unit Euclidean ball with with area density $(2/(1-r^2))^{n-1}$ and volume density $(2/(1-r^2))^n$. We are now able to establish the equivalence between hyperbolic space with density and the Euclidean ball with density:

\begin{prop}\label{equiv}
$\mathbb{H}^n$ with radial volume density $\phi(R)$ and perimeter density $$2\phi(R)\cosh^2{(R/2)}$$is equivalent to the open unit Euclidean ball with density $$\phi(2\tanh^{-1}{r})\bigg(\frac{2}{1-r^2}\bigg)^{n},$$
where $R$ is hyperbolic distance from the origin and $r$ is Euclidean distance from the origin.
\end{prop}

\begin{proof}
Note that $\cosh^2(R/2)=1/(1-\tanh^2(R/2))$. So, the perimeter density may be written as
$$\phi(R)\cdot\frac{2}{1-\tanh^2{(R/2)}}.$$ By Lemma \ref{tanh}, we may rewrite these functions in Poincar\'{e} coordinates to see that the volume density becomes $\phi(2\tanh^{-1}{r})$ and the perimeter density becomes $$\phi(2\tanh^{-1}{r})\cdot\frac{2}{1-r^2},$$ where $r$ is the radial Poincar\'{e} coordinate. 
Finally, let $dV_0$ and $dA_0$ denote the Euclidean volume and area elements. Then in Poincar\'{e} coordinates, the weighted volume element is $$\phi(2\tanh^{-1}{r})\cdot\bigg(\frac{2}{1-r^2}\bigg)^ndV_0$$ and the weighted area element is $$\phi(2\tanh^{-1}{r})\cdot\bigg(\frac{2}{1-r^2}\bigg)^ndA_0.$$ These volume and area elements are equivalent to those on the described weighted Euclidean ball.

\end{proof}

Having demonstrated this equivalence between $\mathbb{H}^n$ and the unit ball with certain densities, we are now ready to make statements about existence and boundedness of isoperimetric regions. First, we need a brief lemma relating area of slices of a region to the perimeter of its boundary.




\begin{lemma}\label{projection}
Let $E$ be a region in either $\mathbb{H}^n$ or the unit ball in $\mathbb{R}^n$ with continuous, nondecreasing radial volume density, with infinite weighted radial distance to the boundary in the ball case. If $E$ has finite weighted volume then for any $r$,
$$P(r)\geq |E_r|.$$

\end{lemma}

\begin{proof}
 Since the radial projection $\pi_r$ of the exterior onto $S_r$ is area nonincreasing, it suffices to show that $\pi_r(B_r)$ covers $E_r$, up to set of measure zero. If not, then a set of positive measure in $E_r$ is not in the projection. The  product of this set with either $(r, \infty)$ or $(r, 1)$, depending on the space, is contained in $E$. But this set must have infinite weighted volume, violating our assumption of finite weighted volume.
\end{proof}

Now we may show the existence of isoperimetric regions in our space. Note that this proof follows a similar structure of an existence proof in Morgan-Pratelli [MP, Thm. 3.3].

\begin{prop}\label{ExistNBall}
Isoperimetric regions exist in $\mathbb{H}^n$ with nondecreasing, continuous radial volume density $g(R)$ and radial perimeter density $g(R)\cdot 2\cosh^2(R/2)$, where $R$ is distance from the origin, for all volumes.
\end{prop}

\begin{proof}
By Proposition \ref{equiv}, it is enough to show that isoperimetric regions exist on the open unit Euclidean ball with radial density $$f(r):=g(2\tanh^{-1}{r})\cdot\bigg(\frac{2}{1-r^2}\bigg)^n.$$
Consider a sequence of regions of the prescribed volume with perimeter tending to the infimum. By compactness [Mo, \textsection 9.1] we may assume convergence of a subsequence to a perimeter-minimizing region. 

The difficulty is that the enclosed volume may be strictly less than the prescribed volume, that some volume disappears to infinity. In that case, for some $\epsilon > 0$, for all $r$, for a tail of the sequence, the volume outside the ball of radius $r$ about the origin is at least $\epsilon$. Fix $R_0$ and such a region $E$, determined by $R_0$, in the tail of the sequence. Using the notation established in section 1, we have

$$ \int_{R_0}^{1}|E_r|f(r)\,dr  \geq \epsilon.$$

Because $f$ satisfies the conditions of Lemma \ref{projection}, we know $|E_r|$ converges to zero as $r$ tends to 1, as it is bounded above by $P(r)$ which must also converge to 0. Hence, when $r$ is chosen close enough to 1, $|E_r|$ is no larger than half the unweighted surface area of the sphere of radius $r$. When this condition holds, we may apply the standard isoperimetric inequality, $$p(r) \geq c|E_r| ^{\frac{n-2}{n-1}},$$ where $c$ is some dimensional constant.

Let $M$ be the supremum of $|E_r|$ for $r\geq R_0$, which is finite by Lemma \ref{projection}. In addition, Lemma \ref{projection} implies that $P_f(R_0)$ satisfies

\begin{equation}\label{cheese}
|\partial E|_f \geq Mf(R_0)
\end{equation} 
where we have used that $f$ is nondecreasing.

Take note that the function $P(r)$ is nonincreasing, so it must be differentiable almost everywhere, and by extension, $P_f(r)$ is diffferentiable almost everywhere. By standard slicing theory [Mo, \textsection4.11], for almost all $r$,

	$$f(r)p(r) \leq -P_f'(r).$$ Since any singular changes in $P_f$ are positive, integration yields that for $R$ large enough

$$|\partial E|_f\geq\int_R^1-P_f'(r)\,dr \geq\int_R^{1} p(r)f(r)dr$$ $$\geq c \int_R^{1}|E_r| ^{\frac{n-2}{n-1}}f(r)dr \geq \frac{c}{M^{\frac{1}{n-1}}}\int_R^{1}|E_r| f(r)dr$$ $$\geq \frac{c}{M^{\frac{1}{n-1}}}\epsilon. $$

By \eqref{cheese},

\begin{equation}\label{egg}
|\partial E|_f^{\frac{n}{n-1}}\geq c f(R_0)^{\frac{1}{n-1}}\epsilon.
\end{equation}
Note that because the perimeters of the regions in the sequence converge, the sequence of perimeters is bounded by some positive $k$. Because $E$ is one of the sequence elements, \eqref{egg} provides that

\begin{equation}
\frac{\epsilon k^{\frac{n}{n-1}}}{c}\geq f(R_0)^{\frac{1}{n-1}}.
\end{equation} This implies that $f$ is bounded above. Recalling the definition of $f$ and noting that $g$ is bounded below provides that $f$ must be unbounded, a contradiction.

Therefore there is no loss of volume to infinity and the limit provides the desired perimeter-minimizing region.

\end{proof}
Proposition \ref{noBoundary} establishes that these isoperimetric regions must be bounded.  See Morgan-Pratelli [MP, Thm. 5.9] for a similar proof in $\mathbb{R}^n$.

\begin{prop}\label{noBoundary}
Isoperimetric regions are bounded in $\mathbb{H}^n$ with nondecreasing, continuous radial volume density $g(R)$ and radial perimeter density $g(R)\cdot 2\cosh^2(R/2)$, where $R$ is distance from the origin.

\end{prop}

\begin{proof}

Denote the volume and perimeter densities by $g$ and $f$.

Suppose we have such an isoperimetric region $E$ which is not contained in any sphere about the origin. We use the notation established at the beginning of this section throughout the proof.

Because the volume density is continuous and nondecreasing, and because $E$ encloses finite weighted volume, Lemma \ref{projection} implies that,  

\begin{equation*}
|E_r|\leq P(r).
\end{equation*} Because the density is nondecreasing,

\begin{equation*}
|E_r|_f\leq P_f(r).
\end{equation*}The standard isoperimetric inequality on the sphere tells us that for any $r$ for which $|E_r|$ is at most half of the area of $S(r)$, one has
\begin{equation}\label{ineq4}
p(r)\geq c|E_r|^{\frac{n-2}{n-1}},
\end{equation} where $c$ is some positive dimensional constant. Because $P(r)$ is bounded by $|\partial E|$, which is finite because $|\partial E|_f$ is finite and $f$ is bounded below, there exists an $R_0>0$ for which $r\geq R_0$ implies $|E_r|$ is at most half of the area of $S(r)$. That is, \eqref{ineq4} holds for all $r$ large enough.
 Because the $f$ density on $E_r$ is constant, we may multiply inequality \eqref{ineq4} by $f(r)=\phi(r)\cdot2\cosh^2{(r/2)}$ to obtain
\begin{equation*}
p_f(r)\geq c|E_r|_f^{\frac{n-2}{n-1}}\phi(r)^{\frac{1}{n-1}}\cosh^{\frac{2}{n-1}}{(r/2)}
\end{equation*}

for $r$ large enough, and where we absorb the power of 2 into constant $c$. Noting that $\phi$ is bounded below by some constant we may absorb this lower bound into a new constant $c$ to obtain

\begin{align}
p_f(r)&\geq c|E_r|_f^{\frac{n-2}{n-1}}\cosh^{\frac{2}{n-1}}{(r/2)},\nonumber \\
&\geq c|E_r|_f\cdot |E_r|_f^{\frac{-1}{n-1}}\cosh^{\frac{2}{n-1}}{(r/2)},\nonumber \\
\label{toUse}
&\geq c|E_r|_f\cdot P_f(r)^{\frac{-1}{n-1}}\cosh^{\frac{2}{n-1}}{(r/2)}.
\end{align}

Note that we have taken negative powers of $|E_r|$, which may be zero if $E$ is disconnected, however for such values of $r$, every quanity in the inequality is zero and so the statements hold trivially. Furthermore, we are justified in taking a negative power of $P_f(R)$ as it is never zero, because $E$ is unbounded.
Because $V(r)$, $P(r)$, and $f$ are monotonic functions, $V_f(r)$ and $P_f(r)$ are differentiable almost everywhere. By Morgan [Mo, \textsection 4.11], \eqref{monkey} and \eqref{ape} hold almost everywhere.

\begin{equation}\label{monkey}
-P_f'(r)\geq p_f(r)
\end{equation} and

\begin{equation}\label{ape}
-2\cosh^2{(r/2)}V_g'(r)=-V_f'(r)=|E_r|_f.
\end{equation}

Hence, for almost every $r$ large enough, \eqref{toUse} may be rewritten using \eqref{monkey} and \eqref{ape} as

\begin{align*}
-P_f'(r)&\geq c(-V_f'(r))P_f(r)^{\frac{-1}{n-1}}\cosh^{\frac{2}{n-1}}{(r/2)},\\
-P_f'(r)P_f(r)^{\frac{1}{n-1}}&\geq -2\cdot cV_g'(r)\cosh^{2\frac{n}{n-1}}{(r/2)},\\
-(P_f^{\frac{n}{n-1}})'(r)&\geq -cV_g'(r)\cosh^{2\frac{n}{n-1}}{(r/2)}
 \end{align*} where $c$ may change from line to line. Because this inequality holds for almost all $r$ large enough, we may perform an integration from any choice of $r$ large enough to $\infty$. The quantity $-(P_f^{\frac{n}{n-1}})(r)$ may not be continuous, but it is increasing, so the integral of this quantity is no larger than the absolute change in $P_f^{\frac{n}{n-1}}$. Knowing that $V(r)$ is continuous provides that the integral of the derivative is exactly the change in $V(r)$. Finally, knowing that $P_f(r)$ and $V_g(r)$ converge to 0 at $\infty$, we may perform this integration from large enough $r$ to $\infty$ and write
\begin{align}
P_f^{\frac{n}{n-1}}(r)&\geq\int_r^{\infty}-(P_f^{\frac{n}{n-1}})'(x)\,dx, \nonumber \\
P_f^{\frac{n}{n-1}}(r)&\geq c\cosh^{2\frac{n}{n-1}}{(r/2)}\int_r^{\infty}-V_g'(x)\,dx,\nonumber \\
P_f(r)&\geq c\cosh^2{(r/2)}(V_g)^{\frac{n-1}{n}}\label{SUP}
\end{align} where the positive constant $c$ may be modified from line to line.

The goal in the remainder of the proof is to derive a contradiction by shifting small pieces of volume of $E$ from far away to places closer to the center. We would like to do this in a way that does not increase the perimeter of the isoperimetric region too dramatically.

On this note, there exists an $R_1$ such that for small positive $\epsilon$ there is an inflated verson of $E$, $E_\epsilon$, that agrees with $E$ outside $B(R_1)$ and has $\epsilon$ more weighted volume satisfying
$$|\partial E_{\epsilon}|_f\leq |\partial E|_f+\epsilon(H+1),$$ where $H$ denotes the constant generalized unaveraged mean curvature of $\partial E$ with respect to densities $f$ and $g$. To understand why this property of $E_{\epsilon}$ may hold, recall that $H$ is the derivative of $f$-weighted perimeter with respect to increases of $g$-weighted volume. So, there is a small number $\bar{\epsilon}$ such that $\epsilon\leq\bar{\epsilon}$ implies 
\begin{equation*}
\Delta P\leq\epsilon(H+1)
\end{equation*}
where $\Delta P$ is understood as the increase in $f$ weighted perimeter of $E$ from adding $g$ weighted volume $\epsilon$  to $E$ somewhere relatively close to the origin. Because $V(r)$ and $g$ are continuous functions, the intermediate value theorem gives the existence of $R_1$ such that $V_g(R_1)=\bar{\epsilon}$.
%

Let $\bar{E}$ denote the restriction of $E_{\epsilon}$ to the ball $B(R)$. Then  $|\bar{E}|_g=|E|_g$. Fixing $R$ large enough, we have by \eqref{SUP},

\begin{align}
|\partial \bar{E}|_f&=|\partial E_{\epsilon}|_f-P_f(R)+|E_R|_f\label{yas}\\
&\leq |\partial E|_f+\epsilon(H+1)-c\cosh^2{(R/2)}\epsilon^{\frac{n-1}{n}}\nonumber\\
&+|E_R|_f.\nonumber
\end{align}

Since $E$ is an isoperimetric region, $|\partial\bar{E}|_f\geq |\partial E|_f$, as they enclose the same volume, and \eqref{yas} becomes
\begin{equation*}
0\leq \epsilon(H+1)-c\cosh^2{(R/2)}\epsilon^{\frac{n-1}{n}}+|E_R|_f.
\end{equation*} Recall that $\epsilon$ is determined by $R$ and can be made arbitrarily small by making $R$ arbitrarily large. For $\epsilon$ very small, the linear term is dwarfed by the power term. Hence, for any $R$ large enough,
\begin{align*}
c\cosh^2{(R/2)}\epsilon^{\frac{n-1}{n}}&\leq |E_R|_f=-2\cosh^2{(R/2)}V_g'(R),\\
c(V_g(R))^{\frac{n-1}{n}}&\leq-V_g'(R),\\
c&\leq-(V_g^{\frac{1}{n}})'(R).\\
\end{align*} This last estimate gives a contradiction with the assumption $V_g(r)>0$ for all $r$, and the proof is complete.

\end{proof}

Now that we know by Propositions \ref{ExistNBall} and \ref{noBoundary} that isoperimetric regions exist and are bounded, we can apply the local analysis in Chambers to conclude they are balls about the origin.

\begin{theorem}\label{weakHyperbolic}

Consider $\mathbb{H}^n$ with smooth, radial, log convex volume density $\phi(R)$ and perimeter density $$\phi(R) \cdot2\cosh^2(R/2),$$
where $R$ is distance from the origin. Then spheres about the origin are uniquely isoperimetric.

\end{theorem}

\begin{proof}


By Proposition \ref{equiv}, we may reduce the problem to the open unit Euclidean ball with volume and perimeter density $\phi(2\tanh^{-1}(r))\cdot(2/(1-r^2))^n$. By Lemma \ref{convex}, this is a log convex density on the open unit ball. Because there is infinite weighted distance to the boundary, Proposition \ref{ExistNBall} guarantees that isoperimetric regions exist. These isoperimetric regions are bounded by Proposition \ref{noBoundary}.

Because this isoperimetric region is bounded in a subset of $\mathbb{R}^n$ with smooth, log convex, radial density, we may apply the same local analysis done by Chambers [Ch] on isoperimetric regions in $\mathbb{R}^n$. From this, we conclude that spheres about the origin are isoperimetric, uniquely because the density is strictly log convex.

\end{proof}

\begin{remark}

The proof of our main Theorem \ref{weakHyperbolic} compared hyperbolic space with perimeter density much larger than volume density to Euclidean space with equal densities. Since the log convex density conjecture for $\mathbb{H}^n$ involves equal densities, the comparison would be with the Euclidean ball with perimeter density much smaller than volume density, about which little is known.

\begin{conj}\label{LCDC} (The log-convex density conjecture on hyperbolic space). In $\mathbb{H}^n$ with smooth log convex radial density, every sphere about the origin is isoperimetric.
\end{conj}

\end{remark}

 \bibliographystyle{alpha}

Lewis \& Clark College 

\url{digiosia@lclark.edu}

\vspace{.5cm}

Williams College 

\url{jih1@williams.edu }

\vspace{.5cm}

Stony Brook University

\url{lea.kenigsberg@stonybrook.edu }

\vspace{.5cm}

Williams College

\url{dsp1@williams.edu }

\vspace{.5cm}

Williams College 

\url{wz1@williams.edu }

\end{document}